\documentclass[12pt]{amsart}
\usepackage{t1enc}
\usepackage[latin2]{inputenc}
\usepackage{verbatim}
\usepackage{amsmath,amsfonts,amssymb,amsthm}
\usepackage[mathcal]{eucal}
\usepackage{enumerate}
\usepackage{graphics}
 \newtheorem{thm}{Theorem}[section]
 \newtheorem{cor}[thm]{Corollary}
 \newtheorem{lem}[thm]{Lemma}
 
 \theoremstyle{definition}
 
 \theoremstyle{remark}
 \newtheorem{rem}[thm]{Remark}
 \numberwithin{equation}{subsection}
 
 \DeclareMathOperator{\mes}{mes}
 \DeclareMathOperator{\supp}{supp}

 \newcommand{\Ne}{\in\mathbb{N}}
 \newcommand{\N}{\in\mathbb{N}}
 \newcommand{\n}{\N}
 \newcommand{\p}{\in\mathbb{P}}

  \newcommand{\m}[1]{\mes\left(#1\right)}

\begin{document}

\title[]{Almost everywhere convergence of Fej\'er means of  two-dimensional triangular Walsh-Fourier series}
\author{Gy\"orgy G\'at}

\address{Institute of Mathematics, University of Debrecen, H-4002 Debrecen, Pf. 400, Hungary}
\email{gat.gyorgy@science.unideb.hu}

\thanks{Research supported by the
Hungarian National Foundation for Scientific Research (OTKA),
grant no.  K111651.}

\date{}

\begin{abstract}
In 1987 Harris proved \cite{har} - among others- that for each $1\le p<2$ there exists a two-dimensional function $f\in L^p$ such that its triangular Walsh-Fourier series diverges almost everywhere. In this paper we investigate the Fej\'er (or $(C,1)$) means of the triangle  two variable Walsh-Fourier series of $L^1$ functions. Namely,
we prove the a.e. convergence
$\sigma_n^{\bigtriangleup}f = \frac{1}{n}\sum_{k=0}^{n-1}S_{k, n-k}f\to f$ ($n\to\infty$)
for each integrable two-variable function $f$.
\end{abstract}

\maketitle
\section{introduction}

In 1971 Fefferman proved \cite{fef} the following result with respect to the trigonometric system.
Let $P$ be an open polygonal region in $\mathbb{R}^2$, containing the
origin. Set
\[
\lambda P = \{(\lambda x^1, \lambda x^2) : (x^1, x^2) \in P \}
\]
for $\lambda > 0$. Then for every $1<p, f\in L^p([-\pi,\pi]^2)$ it holds the relation
\[
\lim_{\lambda\to\infty}\sum_{(n^1,n^2)\in \lambda P}\hat f(n^1, n^2)\exp(\imath (n^1y^1+n^2y^2)) = f(y^1, y^2) \quad \mbox{for a.e.} \quad  y=(y^1,y^2)\in [-\pi,\pi]^2.
\]
That is, $S_{\lambda P}f\to f$ a.e. Sj\"olin gave \cite{sjo} a better result in the case when $P$ is a rectangle. He proved the a.e. convergence for the wider class $f\in L ( \log L)^3 \log \log L$ and for functions  $f\in L ( \log L)^2 \log \log L$ when $P$ is a square. This result for squares is improved by Antonov \cite{ant}. Verifying the result of Sj\"olin with even one more $\log$. That is, for functions $f\in L ( \log L)^2 \log \log \log L$.
There is a sharp constrast between the trigonometric and the Walsh case. In 1987 Harris proved \cite{har} for the Walsh system that
if $S$ is a region in $[0,\infty) \times [0,\infty )$ with piecewise $C^1$ boundary not
always parallel to the axes and $1\le p < 2$, then there exists an $f \in L^p$ such
that $S_{\lambda P}f$ diverges a.e. and in $L^p$ norm as $\lambda\to\infty$.
These results justify the investigation of the Fej\'er (or $(C,1)$) means of triangular sums of two-dimensional Fourier series defined as (see e.g. \cite{gw}):
\[
\sigma^{\bigtriangleup}_n f := \frac{1}{n}\sum_{k=0}^{n-1}S_{k}^{\bigtriangleup}f,
\]
where the triangular partial sums $S_{k}^{\bigtriangleup}f$  defined as
\[
S_{k}^{\bigtriangleup}f(x^1,x^2):=\sum_{i=0}^{k-1}\sum_{j=0}^{k-i-1}\hat{f}
(i,j) \omega_{i}(x^1) \omega_{j}(x^2).
\]
That is, $S_{k}^{\bigtriangleup}f$ is nothing else but $S_{k\Delta}f$, where $\Delta$ is the triangle with vertices $(0,0), (1,0)$ and $(0,1)$.
For the trigonometric system Herriot proved \cite{her} the a.e. (and norm) convergence $\sigma^{\bigtriangleup}_n f \to f $  ($f\in L^1$).
The aim of this paper is verify this result with respect to the Walsh system. The main difficulty is that in the trigonometric case we have a
a simple closed formula for the  kernel functions of this triangular means and this is not the case in the Walsh situation.

Next, we give a brief introduction to the theory of the Walsh-Fourier series.

Let ${\mathbb P}$ denote the set of positive integers, ${\mathbb N} := {\mathbb P}
\cup \{0\}$, and $I:= [0,1)$. For any set $E$ let $E^2$ the cartesian product
$E\times E$. Thus ${\mathbb N}^2$ is the set of integral lattice points in the
first quadrant and $I^2$ is the unit square.
Let $E^1 = E$ and fix $j=1$ or $2$. Denote the $j$ -dimensional Lebesgue
measure of any set $E \subset I^j$ by $\mes(E)$.
Denote the $L^p(I^j)$ norm of any function $f$ by $\|f\|_p$ $(1\le p\le
\infty)$.

Denote the dyadic expansion of $n\in {\mathbb N}$ and $x\in I$
by $n = \sum _{j=0}^{\infty}n_j2^j$ and $x =  \sum_{j=0}^{\infty}x_j2^{-j-1}$
(in the case of $x=\frac{k}{2^m} \, \, k,m\n$ choose the expansion which
terminates  in zeros). $n_i, \, x_i$ are  the $i$-th coordinates of $n, \, x$,
respectively. Set $e_i:= 1/2^{i+1}\in I$, the $i$ th coordinate of $e_i$ is
$1$, the rest are zeros ($i\in \mathbb N$).
Define the dyadic addition $+$ as
\[
x+y = \sum_{j=0}^{\infty}|x_j - y_j|2^{-j-1}.
\]
The sets  $I_n(x) := \{y\in I : y_0=x_0, ... ,y_{n-1}=x_{n-1}\}$ for  $x\in I$,
$I_n:= I_n(0) \, $ for $n\p$ and $I_0(x) := I$ are the dyadic intervals of
$I$.
The set of the  dyadic intervals on $I$ is denoted by  ${\mathcal I} := \{I_n(x) :
x\in I , n\in {\mathbb N}\}$. Denote by
${\mathcal A}_n$ the $\sigma $ algebra
generated by the sets $I_n(x) \, (x\in I)$ and $E_n$ the
conditional expectation operator with respect to ${\mathcal A}_n \,
(n\in {\mathbb N})$. $C$ denotes a constant which may be different from line to line.

For $t=(t^1,t^2)\in I^2$, $b=(b^1,b^2)\n^2$ set the two-dimensional dyadic
rectangle, i.e.  two-dimensional dyadic interval
\[
I_b(t):= I_{b^1}(t^1)\times I_{b^2}(t^2).
\]

For $n=(n^1,n^2)\n^2$ denote by $E_n = E_{n^1,n^2}$ the two-dimensional
expectation operator with respect to the $\sigma$ algebra
${\mathcal A}_n={\mathcal A}_{n^1,n^2}$
generated by the two-dimensional rectangles
$I_{n^1}(x^1)\times I_{n^2}(x^2)\, (x=(x^1, x^2)\in I^2)$.
For $n\in \mathbb P$ denote by $|n|:= \max (j\n : n_j\not= 0)$, that is,
$2^{|n|} \le n < 2^{|n|+1}$.
The Rademacher functions on $I$ are defined as:
\[
r_n(x) := (-1)^{x_n} \quad (x\in I, \, n\n).
\]
The Walsh-Paley system (on $I$) is defined as the sequence of the Walsh-Paley
functions:
\[
\omega_n(x) := \prod_{k=0}^{\infty}(r_k(x))^{n_k} =
(-1)^{\sum_{k=0}^{|n|}n_kx_k}, \quad (x\in I, \, n\n).
\]
That is, ${ \omega} := (\omega_n, n\n)$.
(For details see Fine \cite{Fin}.) We also use the notations $n_{(k)} := \sum_{j=0}^k n_j2^j, n^{(k)} := \sum_{j=k}^{\infty} n_j2^j$.

Consider the Dirichlet and the Fej\'er  kernel functions:
\[
\begin{split}
& D_n:= \sum_{k=0}^{n-1}\omega_k, \\
&  K_n:= \frac{1}{n}\sum_{k=0}^{n-1} D_k,\\
& D_0, K_0:= 0.
\end{split}
\]
The Fourier coefficients, the $n$-th partial sum of the Fourier series, the
$n$-th $(C,1)$ mean of $f\in L^1(I)$:
\[
\begin{split}
& \hat f(n):= \int_I f(x)\omega_n(x)\, dx \, \, (n\n),
\\
& S_nf(y) := \sum_{k=0}^{n-1}\hat f(k)\omega_k(y) =
\int_I f(x+y)D_n(x)\, dx =: f*D_n(y),\\
& \sigma_nf(y) := \frac{1}{n}\sum_{k=0}^{n-1} S_kf(y) = \int_I f(x+y) K_n(x) \, dx =: f*K_n(y),
\, (n\p, y\in I).
\end{split}
\]
Moreover, for $n\n$ we have (\cite[page 7]{sws})
\[
D_{2^n}(x) =\begin{cases} 2^n , \text{ if } x\in I_n,\\ 0 , \mbox{ otherwise }
\end{cases}
\]
and for $n\p$ (\cite[page 28]{sws})
\begin{equation}\label{dirichletkernel}
D_{n}(x) =\omega_n(x) \sum_{i=0}^{|n|}n_ir_i(x)D_{2^i}(x).
\end{equation}

Then, this gives $S_{2^n}f(y) = 2^n\int_{I_n(y)}f(x) dx = E_nf(x)\, (n\n)$.
We say that an operator $T : L^1(I^j) \to L^0(I^j)$ ($L^0(I^j)$ is the space of
measurable functions on  $I^j$) is of type $(L^p,L^p)$ (for
$1\le p \le \infty$) if
$\|Tf\|_p \le C_p\|f\|_p$ with some constant $C_p$ depending
only on $p$ for all $f\in L^p(I^j)$. We say that $T$ is of weak type $(L^1,L^1)$ if
$\mes\{|Tf|>\lambda\} \le C\|f\|_1/\lambda$ for all $f\in L^1(I^j)$ and
$\lambda >0$ ($j=1,2$).
The two-dimensional Walsh-Paley functions, Dirichlet, Fej\'er and Marcinkiewicz kernels are defined as follows:
\[
\begin{split}
&\omega_n(x) := \omega_{n^1}(x^1)\omega_{n^2}(x^2), \quad
D_{n}(x) := D_{n^1}(x^1)D_{n^2}(x^2), \\
& K_{n}(x) := K_{n^1}(x^1)K_{n^2}(x^2), \quad
M_n(x) := \frac{1}{n}\sum_{k=0}^{n-1} D_{k,k}(x).
\end{split}
\]
Moreover, the two-dimensional Fourier coefficients, the $n$-th ($n\Ne^2$) rectangular partial sum of the Fourier series, the $n$-th ($n\p^2$) $(C,1)$ mean and the $n$-th ($n\p$) Marcinkiewicz mean of $f\in L^1(I^2)$:
\[
\begin{split}
& \hat f(n):= \int_{I^2} f(x)\omega_n(x)\, dx \, \, (n\n^2),
\\
& S_nf(y) := \sum_{k^1=0}^{n^1-1}\sum_{k^2=0}^{n^2-1}\hat f(k^1, k^2)\omega_{(k^1, k^2)}(y) =
\int_{I^2} f(x+y)D_n(x)\, dx ,\\
& \sigma_nf(y) := \frac{1}{n^1 n^2} \sum_{k^1=0}^{n^1-1}\sum_{k^2=0}^{n^2-1}S_kf(y) = \int_{I^2} f(x+y) K_n(x) \, dx
\, (n\p^2, y\in I^2),\\
& t_nf(y) := \frac{1}{n}\sum_{k=0}^{n-1}S_{k,k}f(y) = \int_{I^2} f(x+y) M_n(x) \, dx \, (n\p, y\in I^2).
\end{split}
\]

Many papers investigate the behavior of the convergence (and some the divergence)  properties of the two dimensional Fej\'er means with respect to the trigonometric or the Walsh system. We mention the papers \cite{JMZ}, \cite{gat_trig} (trigonometric) and \cite{msw}, \cite{gatproc} (Walsh-Paley system).
This is another story and also very interesting to discuss the almost everywhere
convergence of the Marcinkiewicz means
$\frac{1}{n}\sum_{j=0}^{n-1}S_{j,j}f$ of integrable functions
with respect to orthonormal systems.
Although, this mean is
defined for two-variable functions, in the view of almost
everywhere convergence there are similarities with the
one-dimensional case. On the one side, the maximal convergence space for two dimensional Fej\'er means (no restriction on the set of indices other than they have to converge to $+\infty$) is $L\log^+ L$ (\cite{gat_trig, gatproc}), and on the other side, for the Marcinkiewicz means we have a.e. convergence for every integrable functions (for the trigonometric, Walsh Paley systems).

We mention that the first result is due to Marcinkiewicz \cite{Mar}. But he proved ``only'' for functions in the space $L\log^+L$ the a.e. relation $t_nf\to f$ with respect to the trigonometric system. The ``$L^1$ result''
for the trigonometric and the Walsh-Paley system see the papers of Zhizhiasvili \cite{zhi} (trigonometric system), Weisz \cite{wmar} (Walsh system) and Goginava \cite{gog2, gog} (Walsh system).
Some of these results  (including the proofs)
can also be found in \cite{w5}.

The triangular partial sums and the triangular Dirichlet kernels of the $2$-dimensional Fourier series are defined as
\[
\begin{split}
& S_{k}^{\bigtriangleup}f(x^1,x^2):=\sum_{i=0}^{k-1}\sum_{j=0}^{k-i-1}\hat{f}
(i,j) \omega_{i}(x^1) \omega_{j}(x^2),\quad
 D_{k}^{\bigtriangleup }(x^1,x^2)
:=\sum_{i=0}^{k-1}\sum_{j=0}^{k-i-1}\omega_{i}(x^1)\omega_{j}(x^2).
\end{split}
\]
The Fej\'er means of the  triangular partial sums of the two-dimensional integrable function $f$ (see e.g. \cite{gw}) are
\[
\sigma^{\bigtriangleup}_n f := \frac{1}{n}\sum_{k=0}^{n-1}S_{k}^{\bigtriangleup}f.
\]

For the trigonometric system Herriot proved \cite{her} the a.e. (and norm) convergence $\sigma^{\bigtriangleup}_n f \to f $  ($f\in L^1$).
His method can not be adopted for the Walsh  system, since for the time being there is no kernel formula
available for these systems. The first result in this a.e. convergence issue of triangular means is due to Goginava and Weisz \cite{gw}.
They proved for the Walsh-Paley system and each integrable function the a.e. convergence relation
$\sigma^{\bigtriangleup}_{2^n} f \to f$. That is, we have the subsequence $(\sigma^{\bigtriangleup}_{2^n})$ of the whole sequence of the triangular mean operators.
This result   for every lacunary sequence $(a_n)$ (that is, $a_{n+1}\ge qa_n, \, q>1$) (instead of $(2^n)$) follows from a result of G\'at \cite{gatMlike}.
The aim of this paper is to extend this result of the author for the whole sequence of natural numbers. That is, the almost everywhere convergence $\sigma^{\bigtriangleup}_{n}f\to f$ for every integrable function $f$.

 To demonstrate an important relation between the triangle kernels and the one dimensional Dirichlet kernels  see  some calculations below.
\[
\begin{split}
& K_{n}^{\bigtriangleup }(x^1,x^2)
= \frac{1}{n}\sum_{k=0}^{n-1}D_{k}^{\bigtriangleup }(x^1,x^2)
 =\frac{1}{n}\sum_{k=1}^{n-1}\sum_{i=0}^{k-1}
\sum_{j=0}^{k-i-1}\omega_{i}(x^1) \omega_{j}(x^2)\\
& =\frac{1}{n}\sum_{k=1}^{n-1}\sum_{i=0}^{k-1}\omega_{i}(x^1) D_{k-i}(x^2)
  =\frac{1}{n}\sum_{k=1}^{n-1}\sum_{i=1}^{k}\omega_{k-i}(x^1) D_{i}(x^2)
 =\frac{1}{n}\sum_{i=1}^{n-1}\sum_{k=i}^{n-1}\omega_{k-i}(x^1) D_{i}(x^2)\\
& =\frac{1}{n}\sum_{i=1}^{n-1}D_{n-i}(x^1) D_{i}(x^2)
 =\frac{1}{n}\sum_{i=1}^{n-1}D_{i}(x^1) D_{n-i}(x^2).
\end{split}
\]

In other words, 

\[
\sigma_n^{\bigtriangleup}f(y) := \frac{1}{n}\sum_{k=0}^{n-1}S_{k}^{\bigtriangleup}f(y) 
 = \int_{I^2} f(x+y)K_{n}^{\bigtriangleup}(x) dx = \frac{1}{n}\sum_{k=0}^{n-1}S_{k,n-k}f(y).
\]

That is, the main aim of this paper is to prove the a.e. convergence
\[
\sigma_n^{\bigtriangleup}f = \frac{1}{n}\sum_{k=0}^{n-1}S_{k,n-k}f \to f
\]
for each integrable two-variable function $f$.

In  paper \cite{gat_gog} we introduced the notion of dyadic triangular-Fej\'er means of
two-dimensional Walsh-Fourier series as follows:
\[
\dot{\sigma}_{n}^{\bigtriangleup }f:=\frac{1}{n}
\sum_{k=0}^{n-1}S_{k,n\oplus k}f,
\]
where $\oplus $ is the dyadic (or logical) addition. That is,
\[
k\oplus n := \sum_{i=0}^{\infty} |k_i-n_i|2^i,
\]
where $k_i, n_i$ are the $i$th coordinate of natural numbers $k,n$ with respect to number system based $2$. Remark that the inverse operation 
of $\oplus$ is also $\oplus$.
In paper \cite[Corollary 1]{gat_gog}  we proved 
for each $f\in L^1$ the a.e. relation
\[
\dot{\sigma}_{n}^{\bigtriangleup }f \to f.
\]

The dyadic (or logical) addition is completely different from the ordinary (or arithmetical) one.
Besides, it seems that the ``arithmetical'' version (that is, the $(C,1)$ means of $S_{k, n-k}f$) is a more difficult situation and maybe that is why, there appeared some partial results earlier. See for instance the result of Goginava and Weisz \cite{gw}: $\frac{1}{2^n}\sum\limits_{i=0}^{2^n-1}S_{i,2^n-i}f\to f$ a.e. for every $f\in L^1$.
The ``arithmetical'' triangular means are a natural analogue of the triangular means with respect to the trigonometric system.

However,  the ``dyadic'' (or ``logical'') triangular means defined as
$\frac{1}{n}\sum\limits_{i=0}^{n-1}S_{i,n\oplus i}f$ (\cite{gat_gog}) are also natural analogue of the triangular means with respect to the trigonometric system but they are different.   None of the two results imply the other and the proofs need different methods.

The main result of this paper is:

\begin{thm}\label{maintheorem}
Let $f\in L^1(I^2)$. Then $\sigma^{\bigtriangleup}_{n}f = \frac{1}{n}\sum_{k=0}^{n-1}S_{k,n-k}f\to f$ almost everywhere as $n\to\infty$.
\end{thm}
The main tool in the proof of Theorem \ref{maintheorem} is the following lemma with respect to the maximal triangle Fej\'er kernel.
 By the help of this lemma we will verify  that the maximal operator $\sigma_*^{\bigtriangleup}$ ($\sigma_*^{\bigtriangleup}f = \sup_n |\sigma_n^{\bigtriangleup}f|$) is quasi-local (for the definition of quasi-locality see e.g. \cite[page 262]{sws}) and consequently it is of weak type $(L^1, L^1)$ and then by the standard density argument Theorem \ref{maintheorem} will be implied.

\begin{lem}\label{sup_kernel_int} For $a\in \mathbb{N}$
\[
\int_{I^2\setminus (I_a\times I_a)}\sup_{n\ge 2^a}|K_n^{\bigtriangleup}(x)| dx \le C.
\]
\end{lem}

\section{more lemmas and proofs}

To prove Lemma \ref{sup_kernel_int} we need a sequence of lemmas. The first one is:
\begin{lem}\label{delta1}
There exists a $0<\delta<1$ such that
\[
\int_{I^2} \sup_{n\le 2^A}\left|\frac{1}{2^A}\sum_{k=0}^{2^A-1}\omega_k(x^1)\omega_{k+n}(x^2)\right| dx \le C\delta^A
\]
for every $A\in\mathbb{N}$.
\end{lem}
\begin{proof}
Recall that we
use the notation $n\oplus k := \sum_{j=0}^{\infty}|n_j-k_j|2^j$ ($n,k\n$). That is, the dyadic addition of natural numbers.
First, we discuss the case $n=2^A$ and then $n<2^A$ will be supposed everywhere. That is, let $n=2^A$ now for a moment.
Since $k<2^A$, then $\omega_{k+2^A} = \omega_k \omega_{2^A}$ and
\[
\begin{split}
& \int_{I^2}\left|\frac{1}{2^A}\sum_{k=0}^{2^A-1}\omega_k(x^1)\omega_{k+2^A}(x^2)\right| dx
 = \int_{I^2}\left|\frac{1}{2^A}\sum_{k=0}^{2^A-1}\omega_k(x^1)\omega_{k}(x^2)\right| dx \\
& = \int_{I^2}\frac{1}{2^A}D_{2^A}(x^1+x^2) dx
 = \frac{1}{2^A}.
\end{split}
\]
That is, case $n=2^A$ is cleared and in the sequel $n<2^A$ is supposed.
The integral to be investigated is not greater than the $L^4(I^2)$ norm of
\[
\sup_{n<2^A}\left|\frac{1}{2^A}\sum_{k=0}^{2^A-1}\omega_k(x^1)\omega_{k+n}(x^2)\right|.
\]
 This $L^4(I^2)$ norm is bounded by
\begin{equation}\label{delta2}
\begin{split}
  & \frac{1}{2^A}\left(\int_{I^2}\sum_{n=0}^{2^A-1}\left|\sum_{k=0}^{2^A-1}\omega_k(x^1)\omega_{k+n}(x^2)\right|^4\right)^{1/4}\\
  & =   \frac{1}{2^A}\Biggl(\sum_{n=0}^{2^A-1}\sum_{k=0}^{2^A-1}\sum_{l=0}^{2^A-1}\sum_{i=0}^{2^A-1}\sum_{j=0}^{2^A-1}
  \int_I \omega_{k\oplus l\oplus i\oplus j}(x^1) dx^1
   \int_I \omega_{(k+n)\oplus (l+n)\oplus (i+n)\oplus (j+n)}(x^2) dx^2\Biggr)^{1/4} \\
  & =   :\frac{1}{2^A}\Biggl(\sum_{n,k,l,i,j\in\{0,\dots, 2^A-1\}}B_{n,k,l,i,j}\Biggr)^{1/4}.
\end{split}
\end{equation}
Investigate the integral $B_{n,k,l,i,j}$. Suppose that it is not zero. Then $k\oplus l\oplus i\oplus j$ should be zero. Thus, $j=k\oplus l\oplus i$.
Similarly,
$(k+n)\oplus (l+n)\oplus (i+n)\oplus (j+n)=(k+n)\oplus (l+n)\oplus (i+n)\oplus ((k\oplus l\oplus i)+n)$ should be zero again.
This follows
\begin{equation}\label{delta3}
(k+n)\oplus (l+n)\oplus (i+n)=  (k\oplus l\oplus i)+n.
\end{equation}
We give an upper bound for the number of quadruples $(n,k,l,i)$ satisfying (\ref{delta3}).

Represent $n,k,l,i$ as $0,1$ sequences of length $A$. Divide every $k,l,n,i$ $0,1$ sequence into blocks with four coordinates (elements) in each block. That is, the first blocks are:
\[
(n_3,n_2,n_1,n_0), \quad (k_3,k_2,k_1,k_0), \quad (l_3,l_2,l_1,l_0), \quad(i_3,i_2,i_1,i_0).
\]
The $s$th block:
\[
\begin{split}
& \alpha_{n,s}:= (n_{4s-1}, n_{4s-2}, n_{4s-3}, n_{4s-4}), \quad \alpha_{k,s} = (k_{4s-1}, k_{4s-2}, k_{4s-3}, k_{4s-4}), \\
& \alpha_{l,s}=(l_{4s-1}, l_{4s-2}, l_{4s-3}, l_{4s-4}), \quad \alpha_{i,s}=(i_{4s-1}, i_{4s-2}, i_{4s-3}, i_{4s-4}) \quad (s=1,\dots \lfloor A/4\rfloor).
\end{split}
\]
Suppose that there exists an $s\in \{1,\dots \lfloor A/4\rfloor\}$ such that
\begin{equation}\label{delta4}
  \alpha_{n,s} = (0,0,1,0), \quad \alpha_{k,s} = (0,0,0,0), \quad \alpha_{l,s} = (0,0,1,0), \quad \alpha_{i,s} = (0,1,0,0).
\end{equation}

When we add $k$ and $n$, then in the $s$th block of $k+n$ we find $(0,0,1,\epsilon_{k,n,s})$, where $\epsilon_{k,n,s}\in\{0,1\}$ depends on smaller indices coordinates of $k$ and $n$.
Similarly, for $l+n$ and $i+n$ we have for their $s$th block:
\[
\alpha_{k+n,s} = (0,0,1,\epsilon_{k,n,s}), \quad \alpha_{l+n,s} = (0,1,0,\epsilon_{l,n,s}), \quad \alpha_{i+n,s} = (0,1,1,\epsilon_{i,n,s}).
\]
This gives for the $s$th block of $(k+n)\oplus (l+n)\oplus (i+n)$: $(0,0,0,\epsilon_{n,k,l,i,s})$, where $\epsilon_{n,k,l,i,s}\in \{0,1\}$. On the other hand,
the $s$th block $k\oplus l\oplus i$ is: $\alpha_{k\oplus l\oplus i, s} = (0,1,1,0)$. This regarding $\alpha_{n,s} = (0,0,1,0)$ gives that
\[
\alpha_{(k\oplus l\oplus i)+n,s} = (1,0,0, \tilde\epsilon_{k,l,i,n,s}),
\]
where $\tilde\epsilon_{k,l,i,n,s}$ is either $0$ or $1$. That is, the $s$th block of $(k+n)\oplus (l+n)\oplus (i+n)$ and $(k\oplus l\oplus i)+n$ is different.
Consequently, (\ref{delta3}) does not hold and $B_{n,k,l,i,j}=0$ ($j=k\oplus l\oplus i$). The number of quadruples $(n,k,l,i)\in \times_1^4 \{0,\dots, 2^A-1\}$
for which there is no block with (\ref{delta4}) is bounded by $(2^{16}-1)^{\lfloor A/4\rfloor} 2^{12}$ ($2^{12}$ occurs if $A$ is of form $4t+3$ ($t\in \mathbb{N}$)). Since for every quadruple $(n,k,l,i)$ we have only one $j$ for $B_{n,k,l,i,j}$, then we have at (\ref{delta2}):
\begin{equation*}
\begin{split}
&  \frac{1}{2^A}\Biggl(\sum_{n,k,l,i,j\in\{0,\dots, 2^A-1\}}B_{n,k,l,i,j}\Biggr)^{1/4} \\
& \le
  \frac{1}{2^A}\left((2^{16}-1)^{\lfloor A/4\rfloor} 2^{12}\right)^{1/4} \le 8 \left(\left(\frac{2^{16}-1}{2^{16}}\right)^{1/16}\right)^A = 8\delta^A,
\end{split}
\end{equation*}
where $0<\delta = \left(\frac{2^{16}-1}{2^{16}}\right)^{1/16} <1$. The proof of Lemma \ref{delta1} is complete.
\end{proof}

\begin{rem}
It can be achieved a (little) better (smaller) constant for $\delta$ then $\left(\frac{2^{16}-1}{2^{16}}\right)^{1/16}$
since not only quadruple blocks
\[
\begin{pmatrix}
& \alpha_{n,s}\\
& \alpha_{k,s}\\
& \alpha_{l,s}\\
& \alpha_{i,s}
\end{pmatrix}
=
\begin{pmatrix}
& 0, 0, 1, 0\\
& 0, 0, 0, 0\\
& 0, 0, 1, 0\\
& 0, 1, 0, 0
\end{pmatrix}
\]
\end{rem}
should be excluded ($s=1,\dots \lfloor A/4\rfloor$) but some more. In a similar way of thinking
if the far right coordinate of $\alpha_{n,s}, \alpha_{k,s}, \alpha_{l,s}, \alpha_{i,s}$ remains $0$, take numbers
$n,k,l,i\in \{0,1,\dots 7\}$ expressed in the binary system, that is, as $0,1$ sequences of length $3$. Then find the quadruples
among $\times_1^4 \{0,1,\dots 7\}$ for which
\[
[n+k \pmod{8}] \oplus [n+l \pmod{8}] \oplus [n+i \pmod{8}]  \not= (k\oplus l\oplus i) + n  \pmod{8}.
\]
However, in the point of view of the proof of the main theorem it is  unimportant and $\delta = \left(\frac{2^{16}-1}{2^{16}}\right)^{1/16}$ is quite ``enough''.
\begin{cor}\label{corF+F}
Let $t^2<s$ be natural numbers and let $k_0,\dots, k_{t^2}\in\{0,1\}$ be fixed. Moreover, let
\[
F_{t^2, s}(x) := \sup_{n<2^s}\left|\sum_{k_{t^2+1},\dots, k_{s-1}\in\{0,1\}}\omega_{k^{(t^2+1)}}(x^1)\omega_{(n+k)^{(t^2+1)}}(x^2)\right|,
 \]
 $x\in I^2$. Then for every two-dimensional square $I_{t^2+1}(u) := I_{t^2+1}(u^1)\times I_{t^2+1}(u^2)$ we have
\[
2^{2t^2}\int_{I_{t^2+1}(u)}F_{t^2, s}(x) dx \le C2^{s-t^2}\delta^{s-t^2},
\]
where constant $0<\delta <1$ comes from Lemma \ref{delta1}.
\end{cor}
\begin{proof}
The proof of Corollary \ref{corF+F} is nothing else but a direct application of Lemma \ref{delta1}. Namely, $F_{t^2, s}(x)$ does not depend on
$x^j_0,\dots x^j_{t^2}\, (j=1,2)$, it depends (with respect to $x$) only on $x^j_{t^2+1}, \dots, x^j_{s-1}$ $(j=1,2)$ (and $x_s^2$ in the case of $(n+k)_s=1$). Therefore, instead of $I_{t^2+1}(u) = I_{t^2+1}(u^1)\times I_{t^2+1}(u^2)$ we may write $I_{t^2+1}\times I_{t^2+1}$.
Moreover,
\[
(n+k)^{(t^2+1)} = n^{(t^2+1)}+k^{(t^2+1)}+ \delta(n_{(t^2)}, k_{(t^2)})2^{t^2+1},
\]
where $\delta :\mathbb{N}^2\to \{0,1\}$. This gives
\[
\begin{split}
& F_{t^2, s}(x) = \sup_{n<2^s}\left|\sum_{k_{t^2+1},\dots, k_{s-1}\in\{0,1\}}\omega_{k^{(t^2+1)}}(x^1)\omega_{(n+k)^{(t^2+1)}}(x^2)\right| \\
& \le
\sup_{n<2^s}\left|\sum_{k_{t^2+1},\dots, k_{s-1}\in\{0,1\}}\omega_{k^{(t^2+1)}}(x^1)\omega_{n^{(t^2+1)}+k^{(t^2+1)}}(x^2)\right|\\
& + \sup_{n<2^s}\left|\sum_{k_{t^2+1},\dots, k_{s-1}\in\{0,1\}}
\omega_{k^{(t^2+1)}}(x^1)
\omega_{n^{(t^2+1)}+k^{(t^2+1)} + 2^{t^2+1}}(x^2)\right|\\
& \le
2\sup_{m\le 2^{s-t^2-1}}\left|\sum_{l_0,\dots, l_{s-t^2-2}\in\{0,1\}}
\omega_{l}(y^1)\omega_{m+l}(y^2)\right|,
\end{split}
\]
where $l_0=k_{t^2+1},\dots l_{s-t^2-2}=k_{s-1}$, $y^j_0=x^j_{t^2+1},\dots, y^j_{s-t^2-2}=x^j_{s-1}$ ($j=1,2$). Then apply Lemma \ref{delta1}:
\[
\int_{[0,1)^2}\sup_{m\le 2^{s-t^2-1}}\left|\sum_{l_0,\dots, l_{s-t^2-2}\in\{0,1\}}\omega_{l}(y^1)\omega_{m+k}(y^2)\right| dy \le C2^{s-t^2}\delta^{s-t^2}.
\]
Consequently,
\[
2^{2t^2+2}\int_{I_{t^2+1}(u)}F_{t^2, s}(x) dx  \le \int_{[0,1)^2}\sup_{m\le 2^{s-t^2-1}}\left|\sum_{l_0,\dots, l_{s-t^2-2}\in\{0,1\}}\omega_{l}(y^1)
\omega_{m+k}(y^2)\right| dy \le C2^{s-t^2}\delta^{s-t^2}.
\]
This completes the proof of Corollary \ref{corF+F}.
\end{proof}
We use the notation $J_k = I_k\setminus I_{k+1}\, (k\in\mathbb{N})$.

\begin{lem}\label{marc} Let $0\le t^1\le t^2 < s$ be integers, $J_t = J_{t^1}\times J_{t^2}$. Then
\[
\int_{J_t}\sup_{n<2^s}\left|\sum_{k=0}^{2^s-1}D_k(x^1) D_{n+k}(x^2)\right| dx \le C(t^2-t^1+1)^3\frac{2^{t^1}}{2^{t^2}}2^s \delta^{s-t^2},
\]
where $0 <\delta <1$ comes from Lemma \ref{delta1}.
\end{lem}
\begin{proof}
Let $x = (x^1, x^2) \in J_t = J_{t^1}\times J_{t^2}$, $t^1\le t^2\le s$. Fix $x^j_{t^2+1},\dots, x^j_{s-1}$ ($j=1,2$) and $k_{t_2+1},\dots, k_{s-1}$. We give a bound for the number of tuples $(x^1_{t^1+1},\dots, x^1_{t^2}, k_0,\dots, k_{t_2})$ for which
\begin{equation}\label{D3}
\sum_{k_0,\dots, k_{t^2}\in\{0,1\}} D_k(x^1) D_{n+k}(x^2) \not= 0,
\end{equation}
where $n<2^s$ is a natural number. Since $x\in J_t$, then by (\ref{dirichletkernel}) we have
\begin{equation}\label{D1}
  D_k(x^1) = \omega_{k^{(t^1+1)}}(x^1)(-1)^{k_{t^1}}\left(\sum_{u=0}^{t^1-1}k_u2^u  -k_{t^1}2^{t^1}\right)
\end{equation}

and
\begin{equation}\label{D2}
D_{n+k}(x^2) = \omega_{(n+k)^{(t^2+1)}}(x^2)(-1)^{(n+k)_{t^2}}\left(\sum_{v=0}^{t^2-1}(n+k)_v2^v  -(n+k)_{t^2}2^{t^2}\right).
\end{equation}

\textbf{Case A.} $x^1_{t^1+1}=\dots = x^1_{t^2-1}=0$. This fact with provided that $x^1\in J_{t^1}$ with will be denoted by $x^1\in J_{t_1,0}$. Then the number of tuples $(x^1_{t^1+1},\dots, x^1_{t^2}, k_0,\dots, k_{t_2})$ (not only for those (\ref{D3}) holds) is bounded by
$C2^{t^2}$. This gives
\begin{equation*}
\begin{split}
& \Bigl| \sum_{k_0,\dots, k_{s-1}\in\{0,1\}} D_k(x^1) D_{n+k}(x^2)\Bigr| \\
 & \le  \Bigl|\sum_{k_0,\dots, k_{t^2}\in\{0,1\}}\omega_{k_{(t^2)}^{(t^1+1)}}(x^1)(-1)^{k_{t^1}+(n+k)_{t^2}}\left(\sum_{u=0}^{t^1-1}k_u2^u -k_{t^1}2^{t^1}\right)\\
&  \left(\sum_{v=0}^{t^2-1}(n+k)_v2^v -(n+k)_{t^2}2^{t^2}\right)
  \sum_{ k_{t^2+1}\dots,k_{s-1}\in\{0,1\}}\omega_{k^{(t^2+1)}}(x^1)\omega_{(n+k)^{(t^2+1)}}(x^2)
 \Bigr|\\
 & \le C 2^{t^2} 2^{t^1} 2^{t^2} \Bigl|\sum_{ k_{t^2+1}\dots,k_{s-1}\in\{0,1\}}\omega_{k^{(t^2+1)}}(x^1)\omega_{(n+k)^{(t^2+1)}}(x^2)
 \Bigr|
 \\
 & \le C 2^{t^1+2t^2}F_{t^2,s}(x).
 \end{split}
\end{equation*}

That is, we used
\begin{equation}\label{F+F}
  \Bigl|\sum_{ k_{t^2+1}\dots,k_{s-1}\in\{0,1\}}\omega_{k^{(t^2+1)}}(x^1)\omega_{(n+ k)^{(t^2+1)}}(x^2)\Bigr|
  \le F_{t^2,s}(x).
\end{equation}
Then by Corollary \ref{corF+F} we have
\begin{equation}\label{D4}
\begin{split}
&\int_{J_{t^1,0}\times J_{t^2}}\sup_{n<2^s}\left| \sum_{k=0}^{2^s-1}D_{k}(x^1)D_{n+k}(x^2) \right| dx\\
&
\le C 2^{t^1} 2^{2t^2}\int_{J_{t^2}\times J_{t^2}}F_{t^2,s}(x) dx
\le C 2^{t^1}2^{s-t^2} \delta^{s-t^2} \le C \frac{2^{t^1}}{2^{t^2}}2^s \delta^{s-t^2}.
\end{split}
\end{equation}

\textbf{Case B.} Suppose that $x^1_{t^1+1}=\dots = x^1_{t^1+i-1}=0, x^1_{t^1+i}=1$ for some $1\le i<t^2-t^1$. This fact with provided that $x^1\in J_{t^1}$  will be denoted by $x^1\in J_{t_1,i}$. In this case we give a bound for the integral of the maximal function (it means $\sup_{n<2^s}$) of the following function on the set $ J_{t^1,i}\times J_{t^2}$.

\begin{equation}\label{D5}
\begin{split}
& B_1:=
\sum_{k_0,\dots, k_{t^2}\in\{0,1\}}\omega_{k_{(t^2)}^{(t^1+1)}}(x^1)
(-1)^{k_{t^1}+(n+k)_{t^2}}
\left(\sum_{u=0}^{t^1-1}k_u2^u -k_{t^1}2^{t^1}\right)
 \left(\sum_{l=0}^{t^1+i-1}(n+k)_l2^l \right)\\
&  \sum_{ k_{t^2+1}\dots,k_{s-1}\in\{0,1\}}\omega_{k^{(t^2+1)}}(x^1)\omega_{(n+k)^{(t^2+1)}}(x^2).
 \end{split}
\end{equation}
Besides, for $x\in J_{t^1}\times J_{t^2}$
\begin{equation}\label{D5b2}
\begin{split}
& B_2:=
\sum_{k_0,\dots, k_{t^2}\in\{0,1\}}\omega_{k_{(t^2)}^{(t^1+1)}}(x^1)
(-1)^{k_{t^1}+(n+k)_{t^2}}
\left(\sum_{u=0}^{t^1-1}k_u2^u -k_{t^1}2^{t^1}\right)\\
& \left(\sum_{l=t^1+i}^{t^2-1}(n+k)_l2^l -(n+k)_{t^2}2^{t^2}\right)
 \sum_{ k_{t^2+1}\dots,k_{s-1}\in\{0,1\}}\omega_{k^{(t^2+1)}}(x^1)\omega_{(n+k)^{(t^2+1)}}(x^2)\\
 & \mbox{and}\\
 & \left|\sum_{k=0}^{2^s-1}D_k(x^1) D_{n+k}(x^2)\right| \le |B_1| + |B_2|.
 \end{split}
\end{equation}
In case B we give a bound for
\[
\int_{J_{t^1,i}\times J_{t^2}}\sup_{n<2^s}|B_1| dx
\]
and in case C with its subcases we do the same for $\sup_{n<2^s}|B_2|$. That is, turn our attention to $B_1$.
Have a look at the sum $(t^1+i<t^2)$ below (a part of the sum at \ref{D5})

\begin{equation}\label{B2}
\sum_{k_{t^1+i}=0}^1 \omega_{k_{(t^2)}^{(t^1+1)}}(x^1)\omega_{(n+k)^{(t^2+1)}}(x^2)(-1)^{k_{t^1}+(n+k)_{t^2}}\left(\sum_{u=0}^{t^1-1}k_u2^u -k_{t^1}2^{t^1}\right)
\sum_{l=0}^{t^1+i-1}(n+k)_l2^l.
\end{equation}

This sum can be different from zero only if $(n+k)^{(t^2+1)}(-1)^{(n+k)_{t^2}}$ depends on $k_{t^1+i}$. This also means that $(n+k)^{(t^2)}$ depends on $k_{t^1+i}$. That is, it changes when $k_{t^1+i}$ changes its value from $0$ to $1$. If this is not the case, then all addends depends on $k_{t^1+i}$ as $r_{t^1+i}^{k_{t^1+i}}(x^1) = (-1)^{k_{t^1+i}}$ and consequently  (\ref{B2}) would be $0$. This would give $B_1=0$.
On the other hand,
if $(n+k)^{(t^2)}$ depends on $k_{t^1+i}\in\{0,1\}$, then when $k_{t^1+i}=0$ we have that $(n+k)_{t^1+i}, (n+k)_{t^1+i+1}, \dots, (n+k)_{t^2-1}$ should be $1$ in order to have a change in $(n+k)^{(t^2)}$ as $k_{t^1+i}$ turns to $1$. That is, when $k$ is increased by $2^{t^1+i}$.

That is,
$n+k = (n+k)_{(t^1+i)} + 2^{t^1+i+1} +\dots + 2^{t^2-1} + (n+k)^{(t^2)}$. Consequently,
$k_{t^1+i+1}, \dots, k_{t^2-1}$ should be unchanged.
This implies that the number of tuples $(k_0,\dots, k_{t^2})$ (for any fixed $n<2^s$) satisfying this property is not more than
$C\frac{2^{t^2}}{2^{t^2-t^1-i}} = C2^{t^1+i}$.

By this fact and by (\ref{F+F}) we get an estimation for $B_1$ at (\ref{D5}):

\begin{equation}\label{D6}
\begin{split}
&  |B_1|\\
 & = \Bigl|\sum_{k_0,\dots, k_{t^2}\in\{0,1\}}\omega_{k_{(t^2)}^{(t^1+1)}}(x^1)
 (-1)^{k_{t^1}+(n+k)_{t^2}}
 \left(\sum_{u=0}^{t^1-1}k_u2^u -k_{t^1}2^{t^1}\right)
 \left(\sum_{l=0}^{t^1+i-1}(n+k)_l2^l \right)\\
&  \sum_{k_{t^2+1}\dots,k_{s-1}\in\{0,1\}}\omega_{k^{(t^2+1)}}(x^1)\omega_{(n+k)^{(t^2+1)}}(x^2)
 \Bigr|\\
 & \le C 2^{t^{1}+i} 2^{t^1} 2^{t^1+i} \Bigl|\sum_{ k_{t^2+1}\dots,k_{s-1}\in\{0,1\}}\omega_{k^{(t^2+1)}}(x^1)\omega_{(n+k)^{(t^2+1)}}(x^2)
 \Bigr| \\
  & \le C 2^{3t^1+2i}F_{t^2,s}(x).
 \end{split}
\end{equation}

Then again by Corollary \ref{corF+F} and by the fact that $F_{t^2,s}(x)$  ($s\in\mathbb N$) does not depend on $x^j_0,\dots, x^j_{t^2}$ ($j=1,2$) we have
\begin{equation}\label{DB}
\begin{split}
&\int_{J_{t^1,i}\times J_{t^2}}\sup_{n<2^s} |B_1| dx
\le C 2^{3t^1+2i} \frac{2^{t^2}}{2^{t^1+i}}  \int_{J_{t^2}\times J_{t^2}}F_{t^2,s}(x) dx
 \le C 2^{2t^1+i-t^2} 2^{s-t^2} \delta^{s-t^2}.
\end{split}
\end{equation}
Moreover,
\begin{equation}\label{DB2}
\begin{split}
&\int_{\left(\cup_{i=1}^{t^2-t^1-1}J_{t^1,i}\right)\times J_{t^2}}\sup_{n<2^s} |B_1| dx
 \le C \sum_{i=1}^{t^2-t^1-1}2^{2t^1+i-t^2} 2^{s-t^2} \delta^{s-t^2}
  \le C 2^{t^1}2^{s-t^2} \delta^{s-t^2}.
\end{split}
\end{equation}

\textbf{Case C.} Suppose that $x^1_{t^1+1}=\dots = x^1_{t^1+i-1}=0, x^1_{t^1+i}=1$ for some $1\le  i<t^2-t^1$. That is, $x^1\in J_{t^1, i}$ again as in case B. Then we give an upper bound for the maximal function of  $|B_2|$ on the set $J_{t^1,i}\times J_{t^2}$:
We use estimation (\ref{F+F}). That is,
\[
\Bigl|\sum_{ k_{t^2+1}\dots,k_{s-1}\in\{0,1\}}\omega_{k^{(t^2+1)}}(x^1)\omega_{(n+ k)^{(t^2+1)}}(x^2)\Bigr|
\le F_{t^2,s}(x).
\]

\begin{equation}\label{C1}
\begin{split}
& |B_2|\\
& = \Biggl|\sum_{k_0,\dots, k_{t^2}\in\{0,1\}} \omega_{k_{(t^2)}^{(t^1+1)}}(x^1)
(-1)^{k_{t^1}+(n+k)_{t^2}}
\left(\sum_{u=0}^{t^1-1}k_u2^u -k_{t^1}2^{t^1}\right)\\
&  \left(\sum_{l=t^1+i}^{t^2-1}(n+k)_l2^l -(n+k)_{t^2}2^{t^2}\right)
\sum_{ k_{t^2+1}\dots,k_{s-1}\in\{0,1\}}\omega_{k^{(t^2+1)}}(x^1)\omega_{(n+ k)^{(t^2+1)}}(x^2)\Biggr|\\
& \le \Biggl|\sum_{l=t^1+i}^{t^2-1}2^l\sum_{k_0,\dots, k_{t^1+i-1}\in\{0,1\}}
(-1)^{k_{t^1}}
\left(\sum_{u=0}^{t^1-1}k_u2^u -{k_{t^1}}2^{t^1}\right)\\
& \sum_{k_{t^1+i},\dots, k_{t^2}\in\{0,1\}}\omega_{k_{(t^2)}^{(t^1+1)}}(x^1)(n+k)_l (-1)^{(n+k)_{t^2}}
\sum_{k_{t^2+1}\dots,k_{s-1}\in\{0,1\}}\omega_{k^{(t^2+1)}}(x^1)\omega_{(n+ k)^{(t^2+1)}}(x^2)\Biggr|\\
&
+ \Biggl| 2^{t^2}\sum_{k_0,\dots, k_{t^1+i-1}\in\{0,1\}}(-1)^{k_{t^1}}\left(\sum_{u=0}^{t^1-1}k_u2^u -k_{t^1}2^{t^1}\right)\\
& \sum_{k_{t^1+i},\dots, k_{t^2}\in\{0,1\}}\omega_{k_{(t^2)}^{(t^1+1)}}(x^1) (n+k)_{t^2}
 \sum_{k_{t^2+1}\dots,k_{s-1}\in\{0,1\}}\omega_{k^{(t^2+1)}}(x^1)\omega_{(n+ k)^{(t^2+1)}}(x^2)\Biggr| \\
&=: |\sum_{l=t^1+i}^{t^2-1}B_{2,l}| + |\sum_2| =:|\sum_1| + |\sum_2|.
\end{split}
\end{equation}

We investigate  $\sum_1$.  $\sum_2$ can be treated in the same way.
That is, in case C we  discuss
\begin{equation}\label{C2}
\sum_{k_{t^1+i},\dots, k_{t^2}\in\{0,1\}}\omega_{k_{(t^2)}^{(t^1+1)}}(x^1)\omega_{(n+k)^{(t^2+1)}}(x^2) (n+k)_l (-1)^{(n+k)_{t^2}},
\end{equation}
i.e. the part in $B_{2,l}$ which depends on $k_{t^1+i},\dots, k_{t^2}$ (with fixed $n$ and fixed other $k_i$'s).

for $x\in J_t$, $x^1_{t^1+1}=\dots = x^1_{t^1+i-1}=0, x^1_{t^1+i}=1$ (for some $1\le i<t^2-t^1$ and $t^1+i\le l <t^2$).
Basically, we give a bound for  the number of tuples $(x^1_{t^1+1},\dots, x^1_{t^2}, k_0,\dots, k_{t_2})$ for which $B_{2,l}$ is not $0$. We have four subcases in investigation
of  $\int_{J_{t^1,i}}\sup_{n<2^s}|\sum_N| dx$ $(N=1,2)$.

\textbf{Case CA.}

\[
\begin{split}
& x^1 \in \{x\in I : x^1_0=\dots = x^1_{t^1-1}, x^1_{t^1}=1, x^1_{t^1+1}=\dots = x^1_{t^1+i-1}=0,  x^1_{t^1+i}=1, \\
& x^1_{t^1+i+1}=\dots = x^1_{l-1}=0, x^1_{l+1}=\dots = x^1_{t^2-1}=0\} =: J_{t^1,i,l}
\end{split}
\]
($x^1_l$ is either $0$ or $1$) for some $1\le i<t^2-t^1$ and $t^1+i\le l\le t^2-1$.

\textbf{Case CB.}
There exists a $1\le j\le t^2-l-1$ such that
\[
\begin{split}
& x^1 \in \{x\in I : x^1_0=\dots = x^1_{t^1-1}, x^1_{t^1}=1, x^1_{t^1+1}=\dots = x^1_{t^1+i-1}=0,  x^1_{t^1+i}=1, \\
& x^1_{t^1+i+1}=\dots = x^1_{l-1}=0, x^1_{l+1}=\dots = x^1_{l+j-1}=0, x^1_{l+j}=1\} =: J_{t^1,i,l,j}
\end{split}
\]

\textbf{Case CC.}

There exists a $1\le j\le t^2-l-1$ and $1\le m \le l-t^1-i$ such that
\[
\begin{split}
& x^1 \in \{x\in I : x^1_0=\dots = x^1_{t^1-1}, x^1_{t^1}=1, x^1_{t^1+1}=\dots = x^1_{t^1+i-1}=0,  x^1_{t^1+i}=1, \\
& x^1_{t^1+i+1}= \dots = x^1_{t^1+i+m-1}=0, x^1_{t^1+i+m}=1,   x^1_{l+1}=\dots = x^1_{l+j-1}=0, x^1_{l+j}=1\} \\
& =: J_{t^1,i,l,j,m},
\end{split}
\]

\textbf{Case CD.} There exists a $1\le m \le l-t^1-i$ such that
\[
\begin{split}
& x^1 \in \{x\in I : x^1_0=\dots = x^1_{t^1-1}, x^1_{t^1}=1, x^1_{t^1+1}=\dots = x^1_{t^1+i-1}=0,  x^1_{t^1+i}=1, \\
& x^1_{t^1+i+1}= \dots = x^1_{t^1+i+m-1}=0, x^1_{t^1+i+m}=1,   x^1_{l+1}=\dots = x^1_{t^2-1}=0\} =: J_{t^1,i,l,m}.
\end{split}
\]

The following inequality shows the structure of the investigation with respect to cases CA, CB, CC and CD and $\sum_1$.
\begin{equation}\label{C_structure}
\begin{split}
&  \int_{(\bigcup_{i=1}^{t^2-t^1-1}J_{t^1,i})\times J_{t^2}}\sup_{n<2^s}|\sum_{l=t^1+i}^{t^2-1}B_{2,l}| dx\\
& \le \sum_{i=1}^{t^2-t^1-1}\sum_{l=t^1+i}^{t^2-1}\int_{J_{t^1,i}\times J_{t^2}}\sup_{n<2^s}|B_{2,l}| dx\\
& \le \sum_{i=1}^{t^2-t^1-1}\sum_{l=t^1+i}^{t^2-1}\int_{J_{t^1,i,l}\times J_{t^2}}\sup_{n<2^s}|B_{2,l}| dx
+ \sum_{i=1}^{t^2-t^1-1}\sum_{l=t^1+i}^{t^2-1}\sum_{j=1}^{t^2-l-1}\int_{J_{t^1,i,l,j}\times J_{t^2}}\sup_{n<2^s}|B_{2,l}| dx\\
& + \sum_{i=1}^{t^2-t^1-1}\sum_{l=t^1+i}^{t^2-1}\sum_{j=1}^{t^2-l-1}\sum_{m=1}^{l-t^1-i}\int_{J_{t^1,i,l,j,m}\times J_{t^2}}\sup_{n<2^s}|B_{2,l}| dx\\
& + \sum_{i=1}^{t^2-t^1-1}\sum_{l=t^1+i}^{t^2-1}\sum_{m=1}^{l-t^1-i}\int_{J_{t^1,i,l,m}\times J_{t^2}}\sup_{n<2^s}|B_{2,l}| dx\\
& =:CA + CB + CC + CD.
\end{split}
\end{equation}

\textbf{Case CA} is easy to check and almost the same as case A. The main difference is that we will have to sum also with respect to $i$:
The number of tuples $(x^1_{t^1+1},\dots, x^1_{t^2}, k_0,\dots, k_{t_2})$ (not only for those (\ref{D3}) holds) bounded by
$C 2^{t^2}$, since the number of corresponding tuples $(x^1_{t^1+1},\dots, x^1_{t^2})$ is not more than $C$.
That is, having a look at (\ref{C1}):
\[
|B_{2,l}| \le 2^l 2^{t^1+i} 2^{t^1} 2^{t^2-t^1-i}F_{t^2,s}(x) \le C2^{t^1+l+t^2}F_{t^2,s}(x).
\]
Consequently,
\[
\int_{J_{t^1,i,l}\times J_{t^2}}\sup_{n<2^s}|B_{2,l}| dx
\le C2^{t^1+l+t^2}\int_{J_{t^2}\times J_{t^2}}F_{t^2,s}(x) dx
\le C2^{t^1+l+t^2}2^{-2t^2}2^{s-t^2}\delta^{s-t^2}.
\]

This immediately gives
\[
\sum_{i=1}^{t^2-t^1-1}\sum_{l=t^1+i}^{t^2-1}\int_{J_{t^1,i,l}\times J_{t^2}}\sup_{n<2^s} |B_{2,l}| dx
\le C(t^2-t^1)2^{t^1-t^2}2^s\delta^{s-t^2}.
\]
Similarly, (also by (\ref{C1}))
\[
\begin{split}
 |\sum_2| \le C 2^{t^2}2^{t^1+i}2^{t^1} 2^{t^2-t^1-i}F_{t^2,s}(x)
 \le C2^{t^1+2t^2}F_{t^2,s}(x),
\end{split}
\]
\[
\sum_{i=1}^{t^2-t^1-1}\sum_{l=t^1+i}^{t^2-1}\int_{J_{t^1,i,l}\times J_{t^2}}\sup_{n<2^s}|\sum_2| dx
\le C(t^2-t^1)^2 2^{t^1-t^2}2^s\delta^{s-t^2}.
\]

In \textbf{case CB} (\ref{C2})  equals with
\[
\sum_{k_{t^1+i}, k_{t^1+i+1}\dots, k_{l+j-1}, k_{l+j+1}, \dots k_{t^2}\in\{0,1\}}
(n+k)_l \sum_{k_{l+j}=0}^1  \omega_{k_{(t^2)}^{(t^1+1)}}(x^1)\omega_{(n+k)^{(t^2+1)}}(x^2)(-1)^{(n+k)_{t^2}}
\]
($(n+k)_l$ does not depend on $k_{l+j}$).
The sum $\sum_{k_{l+j}=0}^1  \omega_{k_{(t^2)}^{(t^1+1)}}(x^1)\omega_{(n+k)^{(t^2+1)}}(x^2)(-1)^{(n+k)_{t^2}}$ can be different from zero only in the case when $(n+k)^{(t^2+1)}(-1)^{(n+k)_{t^2}}$ changes as
$k_{l+j}$ turns from $0$ to $1$. That is, when $k_{l+j}=0$ we have that $(n+k)_{l+j}, (n+k)_{l+j+1}, \dots, (n+k)_{t^2-1}$ should be $1$ in order to have a change in $(n+k)^{(t^2+1)}(-1)^{(n+k)_{t^2}}$ as $k_{l+j}$ turns to $1$. That is, when $k$ is increased by $2^{l+j}$. That is,
$n+k = (n+k)_{(l+j-1)} + 2^{l+j} +\dots + 2^{t^2-1} + (n+k)^{(t^2)}$. This implies that the number of tuples $(k_0,\dots, k_{t^2})$ of this kind is not more than
$C\frac{2^{t^2}}{2^{t^2-l-j}} = C2^{l+j}$. Consequently, then number of tuples $(x^1_{t^1+1},\dots, x^1_{t^2}, k_0,\dots, k_{t_2})$ for which $B_{2,l}$ is not zero  is bounded by
$C 2^{t^2-l-j}2^{l+j} = C2^{t^2}$ for a fixed $i<t^2-t^1$.
That is, by (\ref{F+F}) and by the definition of $B_{2,l}$ at (\ref{C1}) we have
\[
|B_{2,l}| \le C2^l 2^{l+j} 2^{t^1} F_{t^2,s}
\]
and
\[
\begin{split}
& \int_{J_{t^1,i,l,j}\times J_{t^2}} \sup_{n<2^s}|B_{2,l}| dx \le C2^{2l+j+t^1} \int_{J_{t^1,i,l,j}\times J_{t^2}} F_{t^2,s}(x) dx \\
& \le C 2^{l+t^1+t^2}\int_{J_{t^2}\times J_{t^2}}F_{t^2,s}(x) dx \le C2^{l+t^1-t^2} 2^{s-t^2}\delta^{s-t^2}.
\end{split}
\]

This immediately gives (have a look at the ``structure'' at (\ref{C_structure}))
\begin{equation}\label{CB_summarizing}
\begin{split}
&  \sum_{i=1}^{t^2-t^1-1}\sum_{l=t^1+i}^{t^2-1}\sum_{j=1}^{t^2-l-1}
\int_{J_{t^1,i,l,j}\times J_{t^2}} \sup_{n<2^s}|B_{2,l}| dx
  \le C\sum_{i=1}^{t^2-t^1-1}\sum_{l=t^1+i}^{t^2-1}\sum_{j=1}^{t^2-l}2^{l+t^1-t^2} 2^{s-t^2}\delta^{s-t^2}\\
  & \le C(t^2-t^1)2^{t^1}2^{s-t^2}\delta^{s-t^2}.
  \end{split}
\end{equation}
The sum $\sum_2$ is discussed later.

Next, we investigate \textbf{case CC}.

We give a bound for $|B_{2,l}(x)|$ for $x\in J_{t^1,i,l,j,m}$ in a way that we find an estimation for the number of tuples $(k_0,\dots, k_{t^2})$
such that (\ref{C2}) is not zero. (If this is not the case, that is, $(k_0,\dots, k_{t^2})$ is so that (\ref{C2}) is  zero, then so does the corresponding addends in $B_{2,l}$.) Change the order of the summations in (\ref{C2}). It equals with
\begin{equation}\label{C3}
\begin{split}
&   \sum_{\substack{k_{t^1+i},\dots, k_{t^1+i+m-1},\\ k_{t^1+i+m+1},\dots, k_{l+j-1},\\  k_{l+j+1}, \dots, k_{t^2}\in\{0,1\}}}
  \prod_{\substack{t^2\ge u\ge t^1+1, \\ u\not= t^1+i+m\\ u\not= l+j}}r_u^{k_u}(x^1)
  \sum_{k_{t^1+i+m}=0}^1 (-1)^{k_{t^1+i+m}}(n+k)_l \\
  & \sum_{k_{l+j}=0}^1 (-1)^{k_{l+j}}\omega_{(n+k)^{(t^2+1)}}(x^2)(-1)^{(n+k)_{t^2}}.
  \end{split}
\end{equation}

We have a fix $n$ and if $k_{t^1+i+m}, k_{l+j}$ runs in $\{0,1\}$ with fixed other indices of $k$, then to avoid (\ref{C3}) to be zero
all the coordinates of $(n+k)_{t^1+i+m+1}, \dots, (n+k)_{l-1}, (n+k)_{l}$ and $(n+k)_{l+1},\dots, (n+k)_{t^2}$ should be $1$ for each $k_i$ $(0\le i\le t^2, i\not=t^1+i+m,
l+j)$
for those addends in $(\ref{C3})$ different from zero. If say, $(n+k)_a=0$ for some $t^1+i+m<a<l$, then as $k_{t^1+i+m}$ changes from $0$ to $1$, we do not have change in
$(n+k)_l$ and in $\omega_{(n+k)^{(t^2+1)}}(x^2)$ and consequently (\ref{C3}) is zero. That is, the number is the tuples
$(k_0,\dots, k_{t^2})$ such that (\ref{C2}) is not zero is not more than $C2^{t^2}2^{-(l-t^1-i-m)}2^{-(t^2-l-j)}=C2^{t^1+i+m+j}$.
Then, by (\ref{F+F}), by the definition of $B_{2,l}$ at (\ref{C1}) and by Corollary \ref{corF+F} we have
\begin{equation}\label{CC1}
\begin{split}
& \int_{J_{t^1,i,l,j,m}\times J_{t^2}}\sup_{n<2^s}|B_{2,l}(x)| dx
\le C\int_{J_{t^1,i,l,j,m}\times J_{t^2}} 2^l2^{t^1}2^{t^1+i+m+j}F_{t^2,s}(x)dx\\
& \le C2^{l-t^1-i-m}2^{t^2-l-j}2^l2^{t^1}2^{t^1+i+m+j}\int_{J_{t^2}\times J_{t^2}}F_{t^2,s}(x)dx\\
& = C 2^{l+t^1-t^2} 2^{2t^2}\int_{J_{t^2}\times J_{t^2}}F_{t^2,s}(x)dx \le C2^{l+t^1-t^2}2^{s-t^2}\delta^{s-t^2}
\end{split}
\end{equation}

This immediately gives
\begin{equation}\label{CC_summarizing}
\begin{split}
&  \sum_{i=1}^{t^2-t^1-1}\sum_{l=t^1+i}^{t^2-1}\sum_{j=1}^{t^2-l-1}\sum_{m=1}^{l-t^1-i}
\int_{J_{t^1,i,l,j,m}\times J_{t^2}} \sup_{n<2^s}|B_{2,l}(x)| dx\\
&  \le C \sum_{i=1}^{t^2-t^1-1}\sum_{l=t^1+i}^{t^2-1}\sum_{j=1}^{t^2-l-1}\sum_{m=1}^{l-t^1-i}
  2^{l+t^1-t^2}2^{s-t^2}\delta^{s-t^2}
\\
  & \le C(t^2-t^1)^32^{t^1}2^{s-t^2}\delta^{s-t^2}.
  \end{split}
\end{equation}
The sum $\sum_2$ is discussed later.

In \textbf{case CD}, in a similar way as above we  give a bound for $|B_{2,l}(x)|$ for $x\in J_{t^1,i,l,m}$. We do it in a way that we find an estimation for the number of tuples $(k_0,\dots, k_{t^2})$
such that (\ref{C2}) is not zero.  Change the order of the summations in (\ref{C2}). It equals with
\begin{equation}\label{C4}
  \sum_{\substack{k_{t^1+i},\dots, k_{t^1+i+m-1},\\ k_{t^1+i+m+1},\dots, k_{t^2}\in\{0,1\}}}
  \prod_{\substack{t^2\ge u>t^1, \\ u\not= t^1+i+m}}r_u^{k_u}(x^1)
  \sum_{k_{t^1+i+m}=0}^1 (-1)^{k_{t^1+i+m}}(n+k)_l \omega_{(n+k)^{(t^2)}}(x^2).
\end{equation}
Remember that $\omega_{(n+k)^{(t^2+1)}}(x^2)(-1)^{(n+k)_{t^2}} = \omega_{(n+k)^{(t^2)}}(x^2)$.
All the coordinates of $(n+k)_{t^1+i+m+1}, \dots, (n+k)_{l-1}, (n+k)_l$  should be $1$
for those addends in $(\ref{C4})$ different from zero. If say, $(n+k)_a=0$ for some $t^1+i+m<a<l$, then as $k_{t^1+i+m}$ changes from $0$ to $1$, we do not have change in
$(n+k)_l$ and in $(n+k)^{(t^2)}$ (and consequently in $\omega_{(n+k)^{(t^2)}}(x^2)$) and this would imply (\ref{C4}) to be zero. That is, the number is the tuples
$(k_0,\dots, k_{t^2})$ such that (\ref{C2}) is not zero is not more than $C2^{t^2}2^{-(l-t^1-i-m)}$.
Then, by (\ref{F+F}), by the definition of $B_{2,l}$ at (\ref{C1}) and by Corollary \ref{corF+F} we have
\begin{equation}\label{CD1}
\begin{split}
& \int_{J_{t^1,i,l,m}\times J_{t^2}}\sup_{n<2^s}|B_{2,l}(x)| dx\\
& \le C\int_{J_{t^1,i,l,m}\times J_{t^2}} 2^l2^{t^1} 2^{t^2}2^{-(l-t^1-i-m)}F_{t^2,s}(x)dx\\
& \le C2^{2t^1+t^2+i+m}2^{l-t^1-i-m}\int_{J_{t^2}\times J_{t^2}}F_{t^2,s}(x)dx\\
& = C 2^{l+t^1-t^2} 2^{2t^2}\int_{J_{t^2}\times J_{t^2}}F_{t^2,s}(x)dx \\
& \le C2^{l+t^1-t^2}2^{s-t^2}\delta^{s-t^2}.
\end{split}
\end{equation}
That is, exactly as in the case CC.  That is (have a look again at (\ref{C_structure})),
\begin{equation}\label{CD_summarizing}
\begin{split}
&   \sum_{i=1}^{t^2-t^1-1}\sum_{l=t^1+i}^{t^2-1}\sum_{m=1}^{l-t^1-i}
\int_{J_{t^1,i,l,j,m}\times J_{t^2}} \sup_{n<2^s}|B_{2,l}(x)| dx\\
&  \le C  \sum_{i=1}^{t^2-t^1-1}\sum_{l=t^1+i}^{t^2-1}\sum_{m=1}^{l-t^1-i}
  2^{l+t^1-t^2}2^{s-t^2}\delta^{s-t^2}
\\
  & \le C(t^2-t^1)^2 2^{t^1}2^{s-t^2}\delta^{s-t^2}.
  \end{split}
\end{equation}
The sum $\sum_2$ can be discussed in the same way. The only difference is that it is more simple. Basically, $\sum_2$ looks like a special $B_{2,l}$ for $l=t^2$. Consequently, there is no cases CB and CC. Only CA and CD cases make sense and these cases has already been investigated. That is, the proof of Lemma \ref{marc} is complete.
\end{proof}

Now, we turn our attention to a lemma concerning the maximal triangular kernel function. This estimation will consist of the three forthcoming lemmas.
First, $s\le t^1\le t^2$ (Lemma \ref{sup_kernel_int_t1}), then the second part will be $t^1 < s\le t^2$ (Lemma \ref{sup_kernel_int_t2}) and the third part will be $t^1\le t^2 < s$ (Lemma \ref{sup_kernel_int_t3}).
Recall that for $k\Ne$ $J_k = I_k\setminus I_{k+1}$ and $n^{(s)}:= \sum_{k=s}^{\infty}n_k2^k$ ($n,s\Ne$). $n^0 =n, n^{|n|+1}=0$.
The first part:
\begin{lem}\label{sup_kernel_int_t1} Let $a\in \mathbb{N}$. Then
\[
\sum_{t^1=0}^{a}\sum_{t^2=t^1}^{\infty}\int_{J_{t^1}\times J_{t^2}} \sup_{A\ge a}\sup_{|n|=A}\frac{1}{2^A}\sum_{s=0}^{t^1}
n_s\left|\sum_{k=0}^{2^s-1}D_{n^{(s+1)}+k}(x^1)D_{n-(n^{(s+1)}+k)}(x^2)\right| dx \le C.
\]
\end{lem}
\begin{proof}
For $t^1\le a, t^2\ge t^1$ and $x\in J_{t^1}\times J_{t^2}$ by the formula for the Dirichlet kernel function (see (\ref{dirichletkernel})) it is clear that
\[
|D_{n^{(s+1)}+k}(x^1)D_{n-(n^{(s+1)}+k)}(x^2)| \le C2^{t^1 + (t^2 \wedge A)},
\]
where $t^2 \wedge A :=\min\{t^2, A\}$. This gives

\[
\begin{split}
& \sum_{t^1=0}^{a}\sum_{t^2=t^1}^{\infty}\int_{J_{t^1}\times J_{t^2}} \sup_{A\ge a}\sup_{|n|=A}\frac{1}{2^A}\sum_{s=0}^{t^1}
\left|\sum_{k=0}^{2^s-1}D_{n^{(s+1)}+k}(x^1)D_{n_{(s)}-k}(x^2)\right| dx \\
& \le C\sum_{t^1=0}^{a}\sum_{t^2=t^1}^{\infty}\int_{J_{t^1}\times J_{t^2}} \sup_{A\ge a}\frac{1}{2^A}\sum_{s=0}^{t^1}
2^{s+t^1+ (t^2\wedge A)} dx\\
&  \le C\sum_{t^1=0}^{a}\sum_{t^2=t^1}^{a}\frac{1}{2^{t^1+t^2}} \sup_{A\ge a} 2^{2t^1+t^2-A} +
 \sum_{t^1=0}^{a}\sum_{t^2=a}^{\infty}\frac{1}{2^{t^1+t^2}}  2^{2t^1}\\
& \le C\sum_{t^1=0}^{a}\sum_{t^2=t^1}^{a}2^{t^1-a} + C\sum_{t^1=0}^{a}\sum_{t^2=a}^{\infty}2^{t^1-t^2} \\
& \le C.
\end{split}
\]
This completes the proof of Lemma \ref{sup_kernel_int_t1}.

\end{proof}

The second part:
\begin{lem}\label{sup_kernel_int_t2} Let $a\n$. Then
\[
\sum_{t^1=0}^{a}\sum_{t^2=t^1}^{\infty}\int_{J_{t^1}\times J_{t^2}} \sup_{A\ge a}\sup_{|n|=A}\frac{1}{2^A}\sum_{s=t^1+1}^{t^2}
n_s\left|\sum_{k=0}^{2^s-1}D_{n^{(s+1)}+k}(x^1)D_{n-(n^{(s+1)}+k)}(x^2)\right| dx \le C.
\]
\end{lem}
\begin{proof}
Since $x\in J_{t^1}$ and $s>t^1$, then $D_{n^{(s+1)}}(x^1)=0$ and consequently $D_{n^{(s+1)}+k}(x^1)=\omega_{n^{(s+1)}}(x^1)D_k(x^1)$. On the other hand,
$n_s=1$ can be supposed and $0\le n - (n^{(s+1)}+k) = n_{(s)} - k <2^s$, $s\le t^2, x^2\in J_{t^2}\subset I_{t^2}$ gives
\[
D_{n_{(s)} - k}(x^2) = n_{(s)} - k.
\]
Thus, also by the help of the Abel transform (for $x^1\notin I_s$ as $x^1\in J_{t^1} = I_{t^1}\setminus I_{t^1+1}, t^1+1\le s$)

\[
\begin{split}
& n_s\left|\sum_{k=0}^{2^s-1}D_{n^{(s+1)}+k}(x^1)D_{n-(n^{(s+1)}+k)}(x^2)\right|
  = n_s \left|\sum_{k=0}^{2^s-1}D_k(x^1)(n_{(s)} - k)\right| \\
& \le
C2^{2s}\left|K_{2^s}(x^1)\right| + \left|\sum_{k=0}^{2^s-1}kD_k(x^1)\right|
 \le C2^{2s}\left|K_{2^s}(x^1)\right| + \left|\sum_{k=0}^{2^s-1}(k+1)K_{k+1}(x^1)\right|.
\end{split}
\]

In \cite{yan} one can find the estimation: If $t^1<s, x^1\in J_{t^1}$, then $|K_{2^s}(x^1)| \le C2^{t^1}$ for $x^1\in I_s(e_{t^1})$ and $|K_{2^s}(x^1)|=0$ for $x^1\notin I_s(e_{t^1})$.
This gives
\[
\begin{split}
&
\sum_{t^1=0}^{a}\sum_{t^2=t^1}^{\infty}\int_{J_{t^1}\times J_{t^2}} \sup_{A\ge a}\sup_{|n|=A}\frac{1}{2^A}\sum_{s=t^1+1}^{t^2\wedge A}
n_s 2^{2s}\left|K_{2^s}(x^1)\right| dx \\
& \le \sum_{t^1=0}^{a}\sum_{t^2=t^1}^{\infty}\int_{I_s(e_{t^1})\times J_{t^2}} \sup_{A\ge a}\sup_{|n|=A}\frac{1}{2^A}\sum_{s=t^1+1}^{t^2\wedge A}
2^{2s+t^1} dx \\
 & \le C \sum_{t^1=0}^{a}\sum_{t^2=t^1}^{\infty}\sup_{A\ge a}\sum_{s=t^1+1}^{t^2\wedge A}2^{s+t^1-t^2-A}\\
 & \le C \sum_{t^1=0}^{a}\sum_{t^2=t^1}^{a}\sup_{A\ge a}2^{t^1-A} +
  C\sum_{t^1=0}^{a}\sum_{t^2=a+1}^{\infty}\sup_{t^2>A\ge a}2^{t^1-t^2}+
 C \sum_{t^1=0}^{a}\sum_{t^2=a+1}^{\infty}\sup_{A\ge t^2}2^{t^1-A}\\
 & \le C\sum_{t^1=0}^{a}(a-t^1+1)2^{t^1-a} + C\sum_{t^1=0}^{a}2^{t^1-a} + C\sum_{t^1=0}^{a}2^{t^1-a} \le C.
\end{split}
\]
We investigate $\left|\sum_{k=0}^{2^s-1}(k+1)K_{k+1}(x^1)\right|$. (Also use the fact that $s\le t^2, A$.)
\[
\begin{split}
&
\sum_{t^1=0}^{a}\sum_{t^2=t^1}^{\infty}\int_{J_{t^1}\times J_{t^2}} \sup_{A\ge a}\sup_{|n|=A}\frac{1}{2^A}\sum_{s=t^1+1}^{t^2\wedge A}
n_s\left|\sum_{k=0}^{2^s-1}(k+1)K_{k+1}(x^1)\right| dx \\
& \le C \sum_{t^1=0}^{a}\sum_{t^2=t^1}^{\infty}\int_{J_{t^1}\times J_{t^2}} \sup_{A\ge a}\frac{1}{2^A}\sum_{s=t^1+1}^{t^2\wedge A}
\left|\sum_{k=0}^{2^{t^1}-1}(k+1)K_{k+1}(x^1)\right| dx \\
& +
C \sum_{t^1=0}^{a}\sum_{t^2=t^1}^{\infty}\int_{J_{t^1}\times J_{t^2}} \sup_{A\ge a}\frac{1}{2^A}\sum_{s=t^1+1}^{t^2\wedge A}
\left|\sum_{k=2^{t^1}}^{2^s-1}(k+1)K_{k+1}(x^1)\right| dx =: A_{\ref{sup_kernel_int_t2}} + B_{\ref{sup_kernel_int_t2}}.
\end{split}
\]
First, discuss $A_{\ref{sup_kernel_int_t2}}$ by $|K_{k+1}(x^1)|\le C2^{t^1}$ ($x\in J_{t^1}$):
\[
\begin{split}
& A_{\ref{sup_kernel_int_t2}} \le C \sum_{t^1=0}^{a}\sum_{t^2=t^1}^{\infty}\frac{1}{2^{t^1+t^2}}\sup_{A\ge a}\sum_{s=t^1+1}^{t^2\wedge A}2^{3t^1-A}\\
& \le C \sum_{t^1=0}^{a}\sum_{t^2=t^1}^{\infty}\sup_{A\ge a}((t^2\wedge A)-t^1)2^{2t^1-A-t^2} \\
& \le C \sum_{t^1=0}^{a}\sum_{t^2=t^1}^{\infty}((t^2\wedge a)-t^1)2^{2t^1-a-t^2} \\
& \le C \sum_{t^1=0}^{a}\sum_{t^2=t^1}^{a}(t^2-t^1)2^{2t^1-a-t^2} + C \sum_{t^1=0}^{a}\sum_{t^2=a+1}^{\infty}(a-t^1)2^{2t^1-a-t^2} \le C.
\end{split}
\]
Next and finally in Lemma \ref{sup_kernel_int_t2}, discuss $B_{\ref{sup_kernel_int_t2}}$.
In \cite{mem} one can find the inequality
\begin{equation}\label{supKnMemi}
  \int_{J_{t^1}}\sup_{n\ge 2^A}|K_n(x^1)| dx^1 \le C\frac{2^{t^1}}{2^A}(A-t^1+1) \quad (A\ge t^1).
\end{equation}
By the help of this inequality we have
\[
\begin{split}
& B_{\ref{sup_kernel_int_t2}} \le C \sum_{t^1=0}^{a}\sum_{t^2=t^1}^{\infty}\sum_{A=a}^{\infty}
\frac{1}{2^A}\sum_{s=t^1+1}^{t^2\wedge A}\sum_{i=t^1}^{s-1}\sum_{k=2^i}^{2^{i+1}-1}
(k+1)\int_{J_{t^1}}|K_{k+1}(x^1)| dx^1 \m{J_{t^2}}\\
& \le C\sum_{t^1=0}^{a}\sum_{t^2=t^1}^{\infty}\sum_{A=a}^{\infty}\frac{1}{2^A}\sum_{s=t^1+1}^{t^2\wedge A}\sum_{i=t^1}^{s-1}\sum_{k=2^i}^{2^{i+1}-1}(k+1)\frac{2^{t^1}}{2^i}(i-t^1+1)2^{-t^2}\\
& \le C\sum_{t^1=0}^{a}\sum_{t^2=t^1}^{\infty}\sum_{A=a}^{\infty}\sum_{s=t^1+1}^{t^2\wedge A}\sum_{i=t^1}^{s-1}2^{i+t^1-t^2-A}(i-t^1+1)\\
& \le C\sum_{t^1=0}^{a}\sum_{t^2=t^1}^{\infty}\sum_{A=a}^{\infty}\sum_{s=t^1+1}^{t^2\wedge A}2^{s+t^1-t^2-A}(s-t^1)\\
& \le C\sum_{t^1=0}^{a}\sum_{t^2=t^1}^{\infty}\sum_{A=a}^{\infty}2^{(t^2\wedge A)+t^1-t^2-A}((t^2\wedge A)-t^1)\\
& \le C\sum_{t^1=0}^{a}\sum_{t^2=t^1}^{a}\sum_{A=a}^{\infty}2^{t^2+t^1-t^2-A}(t^2-t^1) +
 C\sum_{t^1=0}^{a}\sum_{t^2=a+1}^{\infty}\sum_{A=a}^{t^2}2^{t^1-t^2}(A-t^1) \\
 & +
 C\sum_{t^1=0}^{a}\sum_{t^2=a+1}^{\infty}\sum_{A=t^2+1}^{\infty}2^{t^1-A}(t^2-t^1)\\
 &  \le C\sum_{t^1=0}^{a}\sum_{t^2=t^1}^{a}2^{t^1-a}(t^2-t^1)
 + C \sum_{t^1=0}^{a}\sum_{t^2=a+1}^{\infty}2^{t^1-t^2}(t^2-t^1)^2
 + C \sum_{t^1=0}^{a}\sum_{t^2=a+1}^{\infty}2^{t^1-t^2}(t^2-t^1)\\
 & C \sum_{t^1=0}^{a}2^{t^1-a}(a-t^1)^2 + C \sum_{t^1=0}^{a}2^{t^1-a}(a-t^1)^2 + C \sum_{t^1=0}^{a}2^{t^1-a}(a-t^1)
 \le C.
\end{split}
\]
This completes the proof of Lemma \ref{sup_kernel_int_t2}.
\end{proof}

The third part is:

\begin{lem}\label{sup_kernel_int_t3} Let $a\n$. Then
\[
\sum_{t^1=0}^{a}\sum_{t^2=t^1}^{\infty}\int_{J_{t^1}\times J_{t^2}} \sup_{A\ge a}\sup_{|n|=A}\frac{1}{2^A}\sum_{s=t^2+1}^{A}
n_s\left|\sum_{k=0}^{2^s-1}D_{n^{(s+1)}+k}(x^1)D_{n-(n^{(s+1)}+k)}(x^2)\right| dx \le C.
\]
\end{lem}
\begin{proof}
First, for fixed $t=(t^1, t^2), s, A$ we discuss the integral
\[
\int_{J_{t^1}\times J_{t^2}} \sup_{|n|=A}
n_s\left|\sum_{k=0}^{2^s-1}D_{n^{(s+1)}+k}(x^1)D_{n-(n^{(s+1)}+k)}(x^2)\right| dx.
\]
This means that $n_s=1$ can be supposed. Otherwise the integral is zero. Since $x^1\in J_{t^1}=I_{t_1}\setminus I_{t^1+1}, s>t^2\ge t^1$, then $D_{n^{(s+1)}+k}(x^1) = D_{n^{(s+1)}}(x^1) + \omega_{n^{(s+1)}}(x^1)D_k(x^1) = \omega_{n^{(s+1)}}(x^1)D_k(x^1)$. We also have $n-(n^{(s+1)}+k)=n_{(s)}-k = n_{(s-1)}+n_s2^s-k = n_{(s-1)}+2^s-k $ and
consequently by $D_0(x^1)=0, D_{2^s}(x^1)=0$
\[
\begin{split}
& n_s\left|\sum_{k=0}^{2^s-1}D_{n^{(s+1)}+k}(x^1)D_{n-(n^{(s+1)}+k)}(x^2)\right| =
\left|\sum_{k=0}^{2^s-1}D_{k}(x^1)D_{n_{(s-1)}+2^s-k}(x^2)\right| \\
& = \left|\sum_{k=0}^{2^s-1}D_{2^s-k}(x^1)D_{n_{(s-1)}+k}(x^2)\right|
= \left|\sum_{k=0}^{2^s-1}D_{k}(x^1)D_{n_{(s-1)}+k}(x^2)\right|.
\end{split}
\]

The last equality is given by $D_{2^s-k} = D_{2^s} - \omega_{2^s-1}D_k$ (see e.g. \cite{gog}) and $D_{2^s}(x^1)=0$.
By Lemma \ref{marc} we have

\begin{equation}\label{sup_kernel_int_t3b}
\begin{split}
&
\int_{J_{t^1}\times J_{t^2}} \sup_{|n|=A}
n_s\left|\sum_{k=0}^{2^s-1}D_{n^{(s+1)}+k}(x^1)D_{n-(n^{(s+1)}+k)}(x^2)\right| dx\\
& \le
\int_{J_{t^1}\times J_{t^2}} \sup_{|n|=A}
\sup_{n_{(s-1)}\in\mathbb{N}}\left|\sum_{k=0}^{2^s-1}D_{k}(x^1)D_{n_{(s-1)}+k}(x^2)\right| dx\\
& =\int_{J_{t^1}\times J_{t^2}}
\sup_{n<2^s}\left|\sum_{k=0}^{2^s-1}D_{k}(x^1)D_{n+k}(x^2)\right| dx\\
& \le C\frac{2^{t^1}}{2^{t^2}}2^s \delta^{s-t^2}(t^2-t^1+1)^3.
\end{split}
\end{equation}

By (\ref{sup_kernel_int_t3b}) it immediately follows (recall that $0<\delta<1$ ``close to $1$'')
\[
\begin{split}
&
\sum_{t^1=0}^{a}\sum_{t^2=t^1}^{\infty}\int_{J_{t^1}\times J_{t^2}} \sup_{A\ge a}\sup_{|n|=A}\frac{1}{2^A}\sum_{s=t^2+1}^{A}
n_s\left|\sum_{k=0}^{2^s-1}D_{n^{(s+1)}+k}(x^1)D_{n-(n^{(s+1)}+k)}(x^2)\right| dx \\
& \le
C\sum_{t^1=0}^{a}\sum_{t^2=t^1}^{\infty}\sum_{A\ge (a\vee t^2)}\frac{1}{2^A}\sum_{s=t^2+1}^{A}
\frac{2^{t^1}}{2^{t^2}}2^s \delta^{s-t^2}(t^2-t^1+1)^3\\
& \le C\sum_{t^1=0}^{a}\sum_{t^2=t^1}^{\infty}\sum_{A\ge (a\vee t^2)}\frac{2^{t^1}}{2^{t^2}} \delta^{A-t^2}(t^2-t^1+1)^3\\
& \le C\sum_{t^1=0}^{a}\left(\sum_{t^2=t^1}^{a} \sum_{A=a}^{\infty}2^{t^1-t^2}\delta^{A-t^2}(t^2-t^1+1)^3
+ \sum_{t^2=a+1}^{\infty}\sum_{A=t^2+1}^{\infty}2^{t^1-t^2}\delta^{A-t^2}(t^2-t^1+1)^3
\right)\\
& \le C\sum_{t^1=0}^{a}\left(\sum_{t^2=t^1}^{a} 2^{t^1-t^2}\delta^{a-t^2}(t^2-t^1+1)^3
+ \sum_{t^2=a+1}^{\infty}2^{t^1-t^2}(t^2-t^1+1)^3
\right)\\
& \le C \sum_{t^1=0}^{a}\delta^{a-t^1} + C \sum_{t^1=0}^{a}2^{t^1-a}(a-t^1+1)^3 \le C.
\end{split}
\]
This completes the proof of Lemma \ref{sup_kernel_int_t3}.
\end{proof}

Then we turn our attention to the main tool of the proof of Theorem \ref{maintheorem}. That is, to
Lemma \ref{sup_kernel_int}.

\noindent
\textit{Proof of Lemma \ref{sup_kernel_int}}
If $a=0$, then $I^2\setminus (I_a\times I_a) = \emptyset$ and we have nothing to prove. That is, we can suppose that $a\ge 1$. We prove the almost everywhere  relation
\[
I^2\setminus (I_a\times I_a) \subset \left(\bigcup_{t^1=0}^{a-1}\bigcup_{t^2=t^1}^{\infty}J_{t^1}\times J_{t^2}\right) \bigcup \left(\bigcup_{t^2=0}^{a-1}\bigcup_{t^1=t^2}^{\infty}J_{t^1}\times J_{t^2}\right) =:J^1\bigcup J^2.
\]
This will be quite easy. Let
$x=(x^1, x^2)\in I^2\setminus (I_a\times I_a)$. Then, either $x^1$ or $x^2$ (or both) is not an element of $I_a$. Say, $x^1\notin I_a$. Then
$x\in J_{t^1}$ for some $t^1<a$. If $x^2\in I_a$ and $x^2\not=0$, then $x\in J^1$. If $x^1\in J_{t^1}$
and $x^2\notin I_a$, then $x^1\in J_{t^1}$ and $x^2\in J_{t^2}$ for some $t^1, t^2<a$. For $t^2\ge t^1$ we have
$x\in J^1$ and for $t^1\ge t^2$ we have $x\in J^2$. This procedure can be done if $x^1, x^2\not=0$. The set of the points $x=(x^1, x^2)$, where either $x^1=0$ or $x^2=0$  is a zero measure set, so this can be supposed and the a.e. relation $I^2\setminus (I_a\times I_a) \subset J^1\times J^2$ is proved. That is, by Lemmas \ref{sup_kernel_int_t1}, \ref{sup_kernel_int_t2}, \ref{sup_kernel_int_t3} and by the formula of $K_n^{\bigtriangleup}$ the proof of Lemma \ref{sup_kernel_int} is  complete.
 \hfill\qed

\begin{cor}\label{L1norm_of_K_n} Let $n\p$. Then
\[
\|K_n^{\bigtriangleup}\|_1 \le C.
\]
\end{cor}

\begin{proof}
By Lemma \ref{sup_kernel_int} we have
\[
\int_{I^2\setminus (I_{|n|}\times I_{|n|})}|K_n^{\bigtriangleup}| \le C.
\]
Besides,
\[
|K_n^{\bigtriangleup}(x)| \le \frac{1}{n}\sum_{k=0}^{n-1}|D_{k}(x^1)| |D_{n-k}(x^2)| \le
C\frac{1}{n}\sum_{k=0}^{n-1} 2^{|n|}\cdot 2^{|n|} \le C2^{2|n|}.
\]
Hence,
\[
\int_{I_{|n|}\times I_{|n|}}|K_n^{\bigtriangleup}| \le C
\]
and this completes the proof of Corollary \ref{L1norm_of_K_n}.
\end{proof}
Now, we can prove that the maximal operator $\sigma^{\bigtriangleup}_*$ is quasi-local (for the definition of quasi-locality see e.g. \cite[page 262]{sws}) and then a bit later the fact that it is of weak type $(L^1, L^1)$. In other words:
\begin{lem}\label{quasi_local}  Let $f\in L^1(I^2)$, $\int_{I^2}f=0$, $\supp f\subset I_a(u^1)\times I_a(u^2)$  for some $u\in I^2$ and $a\Ne$. Then
\[
\int_{I^2\setminus (I_a(u^1)\times I_a(u^2))}\sigma_*^{\bigtriangleup}f(x) dx \le C\|f\|_1.
\]
\end{lem}

\begin{proof}
From the shift invariancy of the Lebesgue measure we can suppose that $u^1=u^2=0$. If $n < 2^a$, then we have
the kernel $K^{\bigtriangleup}_n(x^1, x^2)$ (which is a linear combination of two-dimensional Walsh-Paley functions $\omega_{j,k}$ with $j,k<2^a$) is $\mathcal A_{a,a}$ measurable. This implies
\[
\sigma^{\bigtriangleup}_nf(y) = \int_{I_a\times I_a} f(x)K^{\bigtriangleup}_n (y+x) dx = K^{\bigtriangleup}_n (y)\int_{I_a\times I_a} f(x) dx =0.
\]
That is, $n\ge 2^a$ can be supposed. By the theorem of Fubini and Lemma \ref{sup_kernel_int} we get
\[
\begin{split}
& \int_{I^2\setminus I^2_a}\sigma_*^{\bigtriangleup}f = \int_{I^2\setminus I^2_a}\sup_{n\ge 2^a}|\sigma_n^{\bigtriangleup}f|
\\ &=
\int_{I^2\setminus I^2_a}\sup_{n\ge 2^a}|\int_{I_a^2} f(x)K^{\bigtriangleup}_n(y+x) dx| dy \\
& \le
\int_{I_a^2} |f(x)|\int_{I^2\setminus I^2_a}\sup_{n\ge 2^a} |K^{\bigtriangleup}_n(z) dz| dx \le C\int_{I_a^2} |f(x)| dx = C\|f\|_1.
\end{split}
\]
This completes the proof of Lemma \ref{quasi_local}.
\end{proof}
\begin{thm}\label{weak1_1}
The operator $\sigma^{\bigtriangleup}_*$ is of weak type $(L^1, L^1)$ and it is also of type $(L^{p}, L^{p})$ for all $1<p\le \infty$.
\end{thm}
\begin{proof}

Now, we know that operator $\sigma^{\bigtriangleup}_*$ is of type $(L^{\infty}, L^{\infty})$ which is given by Corollary \ref{L1norm_of_K_n} and it is quasi-local (Lemma \ref{quasi_local}). Consequently, to prove that operator $\sigma^{\bigtriangleup}_*$ is of weak type $(L^1, L^1)$ is nothing else but to follow the standard argument (see e.g. \cite{sws}). Finally, the interpolation lemma of Marcinkiewicz (see e.g. \cite{sws}) gives that it is also of type $(L^{p}, L^{p})$ for all $1<p\le \infty$.
\end{proof}

\noindent
\textit{Proof of Theorem \ref{maintheorem}. }Next, we turn our attention to the proof of the theorem of convergence, that is, Theorem \ref{maintheorem}.
This is also a trivial consequence of the fact that the maximal operator $\sigma^{\bigtriangleup}_*$ is of weak type $(L^1, L^1)$ and the fact that Theorem \ref{maintheorem}
holds for each two-dimensional Walsh-Paley polynomial (which is also very easy to see). By the standard density argument the proof of Theorem \ref{maintheorem} is complete. \hfill\qed

\end{document}